\documentclass[article]{amsart}
\usepackage{cases}
\usepackage{amsmath, amsfonts, pifont, amssymb}
\usepackage{verbatim}
\usepackage[bookmarks=true]{hyperref}
\usepackage{xcolor}
\usepackage[margin=1.3 in]{geometry}
\newtheorem{theorem}{Theorem}[section]
\newtheorem{lemma}[theorem]{Lemma}
\newtheorem{proposition}[theorem]{Proposition}

\newcommand{\e}{\varepsilon}
\newcommand{\beq}{\begin{equation}}
\newcommand{\eeq}{\end{equation}}
\newcommand{\beqq}{\begin{equation*}}
\newcommand{\eeqq}{\end{equation*}}

\theoremstyle{definition}

\theoremstyle{remark}
\newtheorem{remark}[theorem]{Remark}

\numberwithin{equation}{section}

\newcommand{\abs}[1]{\lvert#1\rvert}
\newcommand{\norm}[1]{\|#1\|}

\usepackage{mathtools}
\numberwithin{equation}{section}

\mathtoolsset{showonlyrefs}
\allowdisplaybreaks[4]
\begin{document}

\title[many-body Schr\"odinger equation on tori]{On Strichartz estimates for many-body Schr\"odinger equation in the periodic setting}
\author{Xiaoqi Huang, Xueying Yu, Zehua Zhao and Jiqiang Zheng}

\address{Xiaoqi Huang
\newline \indent Department of Mathematics, University of Maryland\newline\indent
Kirwan Hall 4306, College Park, MD.}
\email{xhuang49@umd.edu}

\address{Xueying  Yu
\newline \indent Department of Mathematics, Oregon State University\indent 
\newline \indent  Kidder Hall 368, Corvallis, OR 97331\indent }
\email{xueying.yu@oregonstate.edu}

\address{Zehua Zhao
\newline \indent Department of Mathematics and Statistics, Beijing Institute of Technology, Beijing, China.
\newline \indent Key Laboratory of Algebraic Lie Theory and Analysis of Ministry of Education, Beijing, China.}
\email{zzh@bit.edu.cn}

\address{Jiqiang Zheng
\newline \indent Institute of Applied Physics and Computational Mathematics\newline\indent
Beijing, 100088, P.R.China.}
\email{zhengjiqiang@gmail.com}



\begin{abstract}

In this paper, we prove Strichartz estimates for many body Schr\"odinger equations in the periodic setting, specifically on tori $\mathbb{T}^d$, where $d\geq 3$. The results hold for both rational and irrational tori, and for small interacting potentials in a certain sense. Our work is based on the standard Strichartz estimate for Schr\"odinger operators on periodic domains, as developed in Bourgain-Demeter \cite{BD}. As a comparison, this result can be regarded as a periodic analogue of Hong \cite{hong2017strichartz} though we do not use the same perturbation method. We also note that the perturbation method fails due to the derivative loss property of the periodic Strichartz estimate.

\end{abstract}

\maketitle

\setcounter{tocdepth}{1}
\tableofcontents

\parindent = 10pt     
\parskip = 5pt

\section{Introduction}
\subsection{Background and Motivations}
We consider the following many-body Schr\"odinger equation in the periodic setting for $N$ particles in $d$ dimensions, with $d \geq 3$ and $N \geq 1$. 
\begin{equation}\label{maineq}
\begin{cases}
(i \partial_t+H_N)u(t,x_1,\ldots,x_N)=0, \\
u(0,x_1,\ldots,x_N)=u_0(x_1,\ldots,x_N) \in L^2_{x_1,\ldots,x_N} ,   
\end{cases}
\end{equation}
where 
\begin{align*}
H_N=\Delta_x-V_N=\sum_{\alpha=1}^{N} \Delta_{x_{\alpha}}-\sum_{1\leq \alpha<\beta \leq N}V(x_{\alpha}-x_{\beta}) . 
\end{align*}
Each particle is denoted by $x_{\alpha} \in \mathbb{T}^{d}$, where $\mathbb{T}=\mathbb{R}/2\pi\mathbb{Z}$, for any $\alpha \in \{1,\ldots,N \}$. The potential $V$ represents the interactions between any two particles.

The integer parameter $N \geq 1$ in the Schr\"odinger equation represents the number of particles in a quantum system, which can often be very large. The interacting potentials of the form $V(x_{\alpha}-x_{\beta})$ represent the interactions between any two particles, which depend on their relative distance. Additionally, $\mathbb{T}^{d}$ denotes $d$-dimensional tori, which can be either rational or irrational.


When $N=1$, the initial value problem \eqref{maineq} reduces to the nonlinear Schr\"odinger equation (NLS) with a potential, which has been extensively studied in the Euclidean case (i.e., when $\mathbb{T}^{d}$ is replaced by $\mathbb{R}^{d}$). This case is also known as the `single-body case', and the research on decay properties has a long history (see the introduction of \cite{hong2017strichartz}, the survey \cite{schlag2005dispersive}, and references therein). In this paper, we focus on the general case where $N\geq 1$, which can involve new difficulties compared to the single-body case, such as the issue of interacting potentials. See \cite{burq2003strichartz,burq2004strichartz} and the references therein regarding Strichartz estimate for the `single-body case'.

The aim of this paper is to explore the Strichartz-type estimates for the $N$-body Schr\"odinger equation \eqref{maineq} in the periodic setting. Previous research on the Euclidean case of \eqref{maineq} has been conducted in \cite{hong2017strichartz}, with additional research on the two-body case in \cite{chong2021global} utilizing the scheme presented in \cite{keel1998endpoint}. The primary goal of this paper is to generalize the Strichartz-type estimates found in \cite{hong2017strichartz} to the periodic case. Additionally, the author is interested in exploring the recent developments in the topic of `dynamics of NLS on tori' by combining both `periodic spaces' and `many-body Schr\"odinger equations', i.e., studying the estimates for many-body Schr\"odinger equations on periodic spaces. In the next paragraph, we will briefly discuss the background of `NLS on tori'.


NLS is highly relevant in the context of nonlinear optics, and in the study of Bose-Einstein condensates, which has been extensively studied in recent decades in various settings, including the Euclidean spaces, torus setting, and waveguide manifolds\footnote{Waveguide manifolds indicate semi-periodic spaces $\mathbb{R}^m \times \mathbb{T}^n$ ($m,n\geq 1$). The dynamics of NLS on waveguide manifolds is a hot topic understudied in recent decades.}.
The well-posedness theory and the long-time behavior in the Euclidean setting (at least in the defocusing scenario)  have been very well understood (see for instance \cite{Iteam1,BD3,KM1}). 
The periodic model is also an important and challenging setting for the study of NLS. Many works have been done and a significant amount of progress has been made in recent years. We refer to \cite{CZZ,CGZ,HP,HTT1,HTT2,IPT3,IPRT3,KV1,yang2023scattering,YYZ,YY20,Z1,Z2,ZhaoZheng} with regard to the torus and waveguide settings and \cite{LYYZ,sire2022scattering,yu2021global} for other dispersive equations on waveguides. At last, we refer to \cite{Cbook,dodson2019defocusing,Taobook} for some classical textbooks on the study of NLS.


As this paper focuses on the estimates and the PDE-level aspects of equation \eqref{maineq}, rather than the mathematical physics level, we will not delve too deeply into the background of many-body problems/equations from a physical perspective. Interested readers can refer to the introductions of \cite{chen2017global, chong2022global, chong2021global, chong2019dynamical, grillakis2013pair, sigal1987n} and the references therein for more information.


To the best of the authors' knowledge, this paper is the first result towards understanding the long-time dynamics of many-body Schr\"odinger equations in the context of periodic spaces (tori).

\subsection{The statement of main results}
Now we are ready to state the main result of this paper, i.e. Strichartz estimate for \eqref{maineq}. We will also discuss some nonlinear results for this model (see Section \ref{5}).

We note that, as in \cite{hong2017strichartz,zhao2023strichartz}, we need to assume some smallness for the potential $V$ and this smallness does not depend on the initial data (only depends on the particle number $N$ and the universal constant). The condition is as follows,
\begin{equation}\label{condition}
    |V(x)|+\left|\big((1-\Delta)^{\frac{\e}{2}} V\big)(x)\right|\le \frac{C}{N^2},
\end{equation}
for some $\e>0$ that can be arbitrarily small. The main theorem reads,
\begin{theorem}\label{strimany}
Assuming $d\geq 3$, we consider \eqref{maineq} and fix a finite time interval $I$. There exists a small number $\epsilon>0$ such that if the interacting potential $V$ satisfies the condition \eqref{condition}, then for $q = \frac{2(d+2)}{d}$ and any $\alpha \in \{1,\ldots,N \}$, we have
\begin{equation}\label{strimany1}
\|  e^{it H_N} u_0 \|_{L^q_{t,x_{\alpha}}(I\times \mathbb{T}^d)L^2_{\hat{x}_{\alpha}} (\mathbb{T}^{d(N-1)})} \lesssim_\e \|u_0\|_{H^{\e}(\mathbb{T}^{Nd})},
\end{equation}
where $x_{\alpha}$ refers  to the $\alpha$-th variable and  $\hat{x}_{\alpha}$ denotes the remaining $N-1$ spatial variables other than the $\alpha$-th variable $x_{\alpha}$, i.e.,
\begin{equation}
   \hat{x}_{\alpha}=(x_1,\ldots , x_{\alpha-1},x_{\alpha+1}, \ldots ,x_N) \in \mathbb{T}^{d(N-1)},
\end{equation}
and
\begin{equation}
   \|u\|_{H^{\e}(\mathbb{T}^{Nd})}=\|(1-\Delta_x)^{\frac{\e}{2}}  u\|_{L^2(\mathbb{T}^{Nd})}\,\,\,\,\text{where}\,\,\,\, x=(x_1, \cdots, x_N)\in \mathbb{T}^{Nd}.
\end{equation}
\end{theorem}
\begin{remark}
Since $q=\frac{2(d+2)}{d}$ is the endpoint exponent for $d$-dimensional Strichartz estimate (see \cite{BD}), Theorem \ref{strimany} can be understood in this way: one considers a certain particle and fix other particles in $L^2$-norms, then the Strichartz estimate in the periodic setting can be recovered for \eqref{maineq}. As shown in the RHS, there is also an $\epsilon$ derivative loss. The estimates when $q>\frac{2(d+2)}{d}$ can be obtained via the interpolation with the mass conservation. 
\end{remark}
\begin{remark}
Theorem \ref{strimany} also includes the special case ($N=1$, one body case): Strichartz estimate for NLS with a potential in the periodic setting. When the potential satisfies certain conditions (such as \eqref{condition}), the standard Strichartz estimate as in \cite{BD} can be recovered, which is also known to be new.
\end{remark}
\begin{remark}
One can view  $(1-\Delta_x)^{\frac{\e}{2}}$ as a spectral multiplier for the Laplace Beltrami operator on $\mathbb{T}^{Nd}$, or equivalently, as a Fourier multiplier acting on the periodic function $u$ in $\mathbb{R}^{Nd}$, see e.g., \cite[Proposition 4]{huang2022sharp} for the proof of equivalence of these two Sobolev norms.
\end{remark}
\begin{remark}
Theorem \ref{strimany} concerns Strichartz norms in this type: $L^q_{t,x_{\alpha}}L^2_{\hat{x}_{\alpha}}$ (i.e. considering a certain particle $x_{\alpha}$ and fix other particles in $L^2$-norms). Another formulation for Strichartz estimate\footnote{See Theorem 1.1 in \cite{hong2017strichartz} for these two different types of Strichartz estimates in the Euclidean setting.} is considering $L^q_{t,x}$-type norms where $x=(x_1,...,x_N)$ is for all particles (i.e. treating all particles equally, without fixing any particle in $L^2$-norm). One can show,
for $q = \frac{2(Nd+2)}{Nd}$, there exists $\epsilon>0$ such that,
\begin{equation}
\|  e^{it H_N} u_0 \|_{L^q_{t,x_{\alpha}}(I\times \mathbb{T}^{Nd})} \lesssim_{\e}
 \|u_0\|_{H^{\e}(\mathbb{T}^{Nd})},
\end{equation}
The proof is similar to Theorem \ref{strimany} with little modifications so we omit it. See the proof of Theorem \ref{strimany} in Section \ref{4} for more details.    
\end{remark}
The main strategy for proving Theorem \ref{strimany} is briefly stated as follows. First, different from the Euclidean case (\cite{hong2017strichartz}), using a modification of \cite{hong2017strichartz} will not give the desired estimate due to the derivative loss of the Strichartz estimate in the periodic setting (see Section \ref{3} for explanations). The proof is based on the Strichartz estimate on tori \cite{BD} and two more elements: the equivalence of the Sobolev-type norms and the observation that the time interval is finite. See Section \ref{4} for more details. The nonlinear applications of Theorem \ref{strimany} will be discussed in Section \ref{5} and we will make a few remarks in Section \ref{rem}. 

At last, we note that the interacting potentials can be defined in the periodic setting (see the appendix in \cite{nam2020derivation} as an example). For convenience, we still use the notation $V(x-y)$ instead of $V(d(x,y))$ where $d$ is the associated metric.

\subsection{Structure of this paper}
The rest of this article is organized as follows. In Section \ref{pre}, we discuss function spaces and some estimates for this model; in Section \ref{3}, we explain that the perturbation method (as in \cite{hong2017strichartz}) fails due to the derivative
loss property for the periodic Strichartz estimate; in Section \ref{4}, we present the proof for Theorem \ref{strimany} (Strichartz estimate); in Section \ref{5}, we discuss the nonlinear applications of the Strichartz estimates (well-posedness); in Section \ref{rem}, we provide some additional remarks on this line of research.
\subsection{Notations}
We write $A \lesssim B$ to say that there is a constant $C$ such that $A\leq CB$. We use $A \simeq B$ when $A \lesssim B \lesssim A $. Particularly, we write $A \lesssim_u B$ to express that $A\leq C(u)B$ for some constant $C(u)$ depending on $u$. We use $C$ for universal constants and $N$ for the number of particles.

We say that the pair $(p,q)$ is $d$-(Strichartz) admissible if
\begin{equation}
    \frac{2}{p}+\frac{d}{q}=\frac{d}{2},\quad 2 \leq p,q \leq \infty \quad (p,q,d)\neq (2,\infty,2).
\end{equation}
Throughout this paper, we regularly refer to the spacetime norms
\begin{equation}
    \|u\|_{L^p_tL^q_x(I_t \times \mathbb{T}^d)}=\left(\int_{I_t}\left(\int_{ \mathbb{T}^d} |u(t,z)|^q dz \right)^{\frac{p}{q}} dt\right)^{\frac{1}{p}}.
\end{equation}
For the tori case, one often chooses the same exponent for $t$ and $x$. As stated in the above theorems,  we will use the following notations for convenience: 
\begin{itemize}
\item 
$x$ refers to the whole spatial variable;
\item
$x_{\alpha}$ denotes  the $\alpha$-th spatial variable;
\item
$\hat{x}_{\alpha}$ represents the remaining $N-1$ spatial variables except the $\alpha$-th variable $x_{\alpha}$.
\end{itemize}

Similar to the Euclidean case, function spaces such as $V^p_{\Delta}$ are also tightly involved. We will discuss them in Section \ref{pre}. (See also \cite{hong2017strichartz}.) 

Again, similar to the Euclidean case, to deal with the interacting potentials, we define the rotation operator $\mathcal{R}_{\alpha \beta}$ with respect to the $\alpha$-th variable $x_{\alpha}$ and the $\beta$-th variable $x_{\beta}$ by
\begin{equation}
\mathcal{R}_{\alpha \beta}(f(x_1, \ldots, x_{\alpha-1},\frac{x_{\alpha}-x_{\beta}}{\sqrt{2}},x_{\alpha+1}, \ldots, x_{\beta-1},\frac{x_{\alpha}+x_{\beta}}{\sqrt{2}},x_{\beta+1},\ldots,x_N))=f(x_1,\ldots, x_N).    
\end{equation}
That is, after applying the operator $\mathcal{R}_{\alpha \beta}$, the function $f$ rotates its $\alpha$-th variable $x_{\alpha}$ and $\beta$-th variable $x_{\beta}$, resulting in the disappearance of the interaction between them.

\subsection*{Acknowledgment} X. Huang is partially supported by an AMS-Simons travel grant. X. Yu is partially supported by NSF DMS-2306429. Z. Zhao was supported by the NSF grant of China (No. 12101046, 12271032) and the Beijing Institute of Technology Research Fund Program for Young Scholars. J. Zheng was supported by NSF grant of China (No. 12271051) and Beijing Natural Science Foundation 1222019. The first author and the third author have learned many-body Schr\"odinger models and related background during their postdoc careers at the University of Maryland. Thus they highly appreciate Prof. M. Grillakis, Prof. M. Machedon and Dr. J. Chong for related discussions, especially the paper of Hong \cite{hong2017strichartz}.

\section{Preliminaries}\label{pre}
In this section, we discuss the Littlewood-Pelay operators and some function spaces for the model \eqref{maineq}. See Section 2 to Section 4 in \cite{hong2017strichartz} for the Euclidean analogue.

We define the Fourier transform on $ \mathbb{T}^d$ as follows:
\begin{equation}
    (\mathcal{F} f)(\xi)= \int_{ \mathbb{T}^d}f(z)e^{-iz\cdot \xi} \, dz,
\end{equation}
where $\xi=(\xi_1,\xi_2,\ldots,\xi_{d})\in  \mathbb{Z}^d$. We also note the Fourier inversion formula
\begin{equation}
    f(z)=c \sum_{(\xi_{1},\ldots,\xi_{d})\in \mathbb{Z}^d}  (\mathcal{F} f)(\xi)e^{iz\cdot \xi}.
\end{equation}
Moreover, we define the Schr{\"o}dinger propagator $e^{it\Delta}$ by
\begin{equation}
    \left(\mathcal{F} e^{it\Delta}f\right)(\xi)=e^{-it|\xi|^2}(\mathcal{F} f)(\xi).
\end{equation}
We are now ready to define the Littlewood-Pelay projections. First, we fix $\eta_1: \mathbb{R} \rightarrow [0,1]$, a smooth even function satisfying
\begin{equation}
    \eta_1(\xi) =
\begin{cases}
1, \ |\xi|\le 1,\\
0, \ |\xi|\ge 2,
\end{cases}
\end{equation}
and $M=2^j$ a dyadic integer. Let $\eta^d :\mathbb{R}^d\rightarrow [0,1]$, 
\begin{align*}
\eta^d(\xi)=\eta_1(\xi_1)\eta_1(\xi_2)\eta_1(\xi_3)\cdots \eta_1(\xi_d)  .
\end{align*}
We define the Littlewood-Pelay projectors $P_{\leq M}$ and $P_{ M}$ by
\begin{equation}
    \mathcal{F} (P_{\leq M} f)(\xi):=\eta^d\left(\frac{\xi}{M}\right) \mathcal{F} (f)(\xi), \quad \xi \in   \mathbb{Z}^d, 
\end{equation}
and
\begin{equation}
P_M f=P_{\leq M}f-P_{\leq \frac{M}{2}}f.
\end{equation}
For any dyadic $M\in (0,\infty)$, we define
\begin{equation}
    P_{\leq M}:=\sum_{L\leq M,L \textmd{ dyadic}}P_L,\quad P_{> M}:=\sum_{M>L,L \textmd{ dyadic}}P_L.
\end{equation}
Next, we state the standard Strichartz estimate in the tori setting. See \cite{BD,KV1}. (Moreover, see \cite{keel1998endpoint} for the Euclidean analogue and \cite{Barron} for the waveguide manifold analogue.)
\begin{lemma}[Strichartz estimate in  tori]\label{toriStrichartz}
 Fix $d \geq 1$ and a finite time interval $I$. For $p \geq \frac{2(d+2)}{d}$, we have
\begin{equation}
\| P_{\leq M} e^{it \Delta_x} u_0 \|_{L^p_{t,x}(I \times \mathbb{T}^d)} \lesssim_{|I|} M^{\frac{d}{2}-\frac{d+2}{p}+\epsilon} \|u_0\|_{L^2_x(\mathbb{T}^d)}.
\end{equation}  
Here the $\epsilon$ can be removed when $p$ is apart from the endpoint $\frac{2(d+2)}{d}$.
\end{lemma}
\begin{remark}
    As we can see, compared with the Euclidean case, there are three main differences for the tori case regarding the standard Strichartz estimate: 1. the appearance of the frequency truncation operator; 2. it is a local estimate (i.e. on finite time interval $I$); 3. the derivative loss $\epsilon$. These differences will cause some new difficulties for our problem compared to the Euclidean case. 
\end{remark}
It is natural to extend the above estimate for the many-body case as follows,
\begin{proposition}[Strichartz estimate for the many-body case]\label{mbStrichartz}
Fix $d \geq 1$, $N\geq 1$ and a finite time interval $I$. For $p \geq \frac{2(d+2)}{d}$ and $\alpha \in \{1,\ldots,N \}$, we have
\begin{equation}
\| P_{\leq M} e^{it \Delta_{x_1,\ldots , x_N} } u_0 \|_{L^p_{t,x_{\alpha}}(I \times \mathbb{T}^d)L^2_{\hat{x}_{\alpha}} (\mathbb{T}^{d(N-1)})}  \lesssim M^{\frac{d}{2}-\frac{d+2}{p}+\epsilon} \|u_0\|_{L^2_x(\mathbb{T}^{Nd})},
\end{equation}
where $x_{\alpha}$ refers  to the $\alpha$-th variable and  $\hat{x}{\alpha}$ denotes the remaining $N-1$ spatial variables other than the $\alpha$-th variable $x_{\alpha}$, i.e.,
\begin{equation}
   \hat{x}_{\alpha}=(x_1,\ldots, x_{\alpha-1},x_{\alpha+1}, \ldots ,x_N) \in \mathbb{T}^{d(N-1)}.
\end{equation}
Here the Littlewood-Pelay operator $P_{\leq M}$ indicates the frequency truncation for all $x_1,\ldots, x_N$ (it is also true if it is only restricted to and $x_{\alpha}$-direction, i.e. $P_{x_{\alpha} \leq M}$). Again, the $\epsilon$ can be removed when $p$ is different from the endpoint $\frac{2(d+2)}{d}$.   
\end{proposition}
\begin{proof}[Proof of Proposition \ref{mbStrichartz}]
We note that the Littlewood-Pelay operator $P_{\leq M}$ commutes with the linear Schr\"odinger operator $e^{it\Delta_x} = e^{it \Delta_{x_1,\ldots ,x_N} }$. Moreover, $e^{it \Delta_{x_1,\ldots ,x_N} }$ is unitary. Applying the Minkowski ($p>2$) and the standard tori Strichartz estimate (Lemma \ref{toriStrichartz}), we have
\begin{align}
\| P_{\leq M} e^{it \Delta_{x_1,\ldots, x_N} } u_0 \|_{L^p_{t,x_{\alpha}}(I \times \mathbb{T}^d)L^2_{\hat{x}_{\alpha}} (\mathbb{T}^{d(N-1)})} &= \| P_{x_{\alpha}, \leq M} 
 e^{it \Delta_{x_{\alpha}}}  u_0 \|_{L^p_{t,x_{\alpha}}(I \times \mathbb{T}^d)L^2_{\hat{x}_{\alpha}} (\mathbb{T}^{d(N-1)})} \\
&\lesssim \| P_{x_{\alpha}, \leq M} 
 e^{it \Delta_{x_{\alpha}}}  u_0 \|_{L^2_{\hat{x}_{\alpha}}L^p_{t,x_{\alpha}}(I \times \mathbb{T}^d)} \\
 &\lesssim M^{\frac{d}{2}-\frac{d+2}{p}+\epsilon} \|u_0\|_{L^2_{\hat{x}_{\alpha}} L^2_{x_{\alpha}}} \\
 &= M^{\frac{d}{2}-\frac{d+2}{p}+\epsilon} \|u_0\|_{L^2_x}.
\end{align}  
Again, here $I$ indicates a finite time interval since the Strichartz estimate in the tori setting (Lemma \ref{toriStrichartz}) concerns the local estimate.
\end{proof}
 As mentioned in \cite{hong2017strichartz} at the end of Section 3.3, Strichartz estimates with frozen spatial variables (as Proposition \eqref{mbStrichartz} above) are insufficient to prove Strichartz estimate for \eqref{maineq} due to the presence of interacting potentials. Therefore, a space-time norm that plays the role of the rotated space-time norm is required. This part is almost the same as Section 4.1 in \cite{hong2017strichartz}, with some natural modifications. For more details, we also refer to \cite{HTT1,HTT2,koch2007priori}.

We note that the definitions and properties in Section 4.1 of \cite{hong2017strichartz} are general enough to be naturally applied to our model in the tori setting. The authors construct function spaces with favorable properties for a separable Hilbert space $H$ and self-adjoint operator $S$. In this paper, we can choose $H$ to be $L^2_x$ and $S$ to be $\Delta_{x}$ in the tori setting, where $x=(x_1,\ldots,x_N)$ and $x_{\alpha}\in \mathbb{T}^d$ for $\alpha \in \{1,\ldots,N \}$, as in \eqref{maineq}. Therefore, the definitions and associated properties for our case will also hold. Hence, we refer to Section 4.1 of \cite{hong2017strichartz} for the function spaces and corresponding estimates and properties. For instance, we will use the following property of the $V^p_{\Delta}$-space, which follows from the definition (see Proposition 2 in \cite{hong2017strichartz}).
\begin{equation}
\|\mathbf{1}_{[0,\infty)]} e^{it\Delta_x} u_0\|_{V^p_{\Delta_x}}= \|u_0\|_{L^2_x}.   
\end{equation}
Moreover, the duality, the inclusion properties, and the transference principle of $V^p_{\Delta}$-space are also often used. See Section 4.1 of \cite{hong2017strichartz}. The transference principle is as follows,
\begin{lemma}[Transference principle]\label{trans}
Let $d \geq 1$, $1<p<2$, $q\geq 2$ and $X$ be a Banach space. If a function $u: \mathbb{R} \rightarrow X$ satisfies the bound 
\begin{equation}
\|e^{it\Delta_x}u_0\|_{L^q_t X} \lesssim \|u_0\|_{L^2_x},   
\end{equation}
then
\begin{equation}
 \|u\|_{L^q_t X} \lesssim \|u\|_{V^p_{\Delta_x}}.    
\end{equation}
\end{lemma}
\begin{remark}
 We note that the Bourgain spaces $X^{s,b}$ (also known as Fourier restriction space) enjoy the similar transfer principle (see \cite{Taobook} for more info.). As summarized in \cite{hong2017strichartz}, the Strichartz estimates in the $V^p_{\Delta_x}$ sharpen the bounds in $X^{s,b}$ by $0+$ in that Strichartz estimates in the $X^{s,b}$ space do not cover the endpoint Strichartz estimates, while those in the $V^p_{\Delta_x}$-space do.   
\end{remark}

\section{Why does the perturbation method fail for the periodic case?}\label{3}
In this section, we will explain why the perturbation method (as in \cite{hong2017strichartz} with suitable modifications) \textbf{fail} for the periodic case. It is natural try this method since the small potentials are under considerations. We will first explain the main idea and attempt to apply this method to prove the Strichartz estimate for \eqref{maineq} (see \cite{hong2017strichartz} for the Euclidean case.). Then we can see how it fails.

The proof of many body Strichartz estimates in \cite{hong2017strichartz} relies on the properties of the function space $V^p_{\Delta_x}$ and a perturbation method. The main idea is to first establish a nonlinear estimate for each arbitrary interacting potential by treating it as a perturbation, and then summing up all the potentials. The key estimate in the proof is Proposition \ref{key} (the tori analogue of Proposition 4 in \cite{hong2017strichartz}), which deals with a single arbitrary interacting potential by regarding it as a forcing term. This estimate allows one to handle all the interacting potentials uniformly as perturbations. Finally, by using the smallness assumption, we can apply a perturbation method to obtain the desired Strichartz estimate. We refer to \cite{hong2017strichartz} for more details.

In comparison to the case with a single potential ($N=1$), the interacting potentials in this problem pose difficulties due to their rotational invariance. As a result, the function space $V^p_{\Delta_x}$, which is flexible under rotations, is required. We consider the Strichartz estimate in Theorem \ref{strimany}. In fact, one may also consider Strichartz estimate in the form of Proposition \ref{mbStrichartz} by replacing the Laplacian by the $H_N$ operator in \eqref{maineq}. (However, that case would be more difficult.\footnote{In contrast to the Euclidean case (\cite{hong2017strichartz}), the presence of a frequency truncation operator in this proof would be a new difficulty. This operator is necessary to handle the low-frequency modes, which are affected differently by the interaction potentials than the high-frequency modes.})

The main estimate one needs is as follows:

\begin{proposition}\label{key}
Let $d \geq 3$, $1<p<2$ and $q > \frac{2(d+2)}{d}$. Let $I$ be a finite interval. Consider $u$ solves \eqref{maineq}. Then, we have
\begin{equation}
\big\| \textbf{1}_{I} \int_0^t e^{i(t-s)\Delta_x}(V(x_{\alpha}-x_{\beta})u(s)) \, ds  \big\|_{V^p_{\Delta}} \leq C \|V\|_{L_x^{\frac{q}{q-2}}(\mathbb{T}^d)} \| u\|_{V^p_{\Delta}} ,
\end{equation}
where $C$ is for the universal constant.
\end{proposition}
\begin{remark}
Proposition \ref{key} indicates that one can regard the potential terms as perturbations. It suffices to consider one arbitrary interacting potential $V(x_{\alpha}-x_{\beta})$ since the $V^p_{\Delta}$-norm is rotation-flexible.
\end{remark}
\begin{remark}
See Proposition 4 in \cite{hong2017strichartz} for the Euclidean analogue. Since one concerns the tori case, it is enough to consider a local-in-time version (on a finite time interval $I$).
\end{remark}
We now explain why this will not work as follows. (We will try to prove it as in the Euclidean case with natural modifications.) The issue comes from the derivative loss property of the Strichartz estimate on tori.
\begin{proof}[A tentative proof which fails]
For notational convenience, we denote
\begin{equation}
    w= \textbf{1}_{I} \int_0^t e^{i(-s)\Delta_{x}}( F(s)) \, ds,
\end{equation}
where $F=V(x_{\alpha}-x_{\beta})u(s)$ is treated as the forcing term (or say a perturbative term).

We will estimate $w$ by the duality argument. Since we only expect $w \in V^p_{-}$, not $w \in V^p$, we consider $\tilde{w}(t)=w(-t)$.

Similar to Proposition 4 of \cite{hong2017strichartz}, using duality, it suffices to show that 
\begin{equation}
    \sum_{j=1}^{J}\langle a(t_{j-1}),\tilde{w}(j)-\tilde{w}(t_{j-1}) \rangle_{L^2_x}  \lesssim \|V\|_{L_x^{\frac{q}{q-2}}} \|u\|_{V^p_{\Delta}}
\end{equation}
for any fine partition of unity $t=\{t_j\}_{j=0}^J$ and any $U^{p'}$-atom $a(t)=\sum_{k=1}^{K} \textbf{1}_{(s_{k-1},s_k)}\phi_{k-1}$. (We note that the $U^{p'}$-space is the dual of the $V^p_{\Delta}$-space.) 

Doing some standard simplifications as in Proposition 4 of \cite{hong2017strichartz} (expanding atoms $a$ in terms of $\phi_k$), one can get a simpler sum
\begin{equation}
    \sum_{k=1}^{K}\langle \phi_{k-1}, \tilde{w}(s_k)-\tilde{w}(s_{k-1})\rangle_{L^2_x}.
\end{equation}
We further write it as
\begin{align}
&\sum_{k=1}^{K}\langle \phi_{k-1}, \tilde{w}(s_k)-\tilde{w}(s_{k-1})\rangle_{L^2_x} \\
=&  -\sum_{k=1}^{K} \int_{-s_k}^{-s_{k-1}} \langle \phi_{k-1},e^{-is\Delta_x}(F(s)) \rangle_{L^2_x} \, ds \\
=&  -\sum_{k=1}^{K} \int_{-s_k}^{-s_{k-1}} \langle e^{is\Delta_x}\mathcal{R} \phi_{k-1},\mathcal{R}( F(s)) \rangle_{L^2_x} \, ds \\
=&  -\sum_{k=1}^{K} \int_{\mathbb{R}} \langle e^{is\Delta_x}\mathcal{R} \phi_{k-1}, \textbf{1}_{[-s_k,-s_{k-1}]} \mathcal{R}( F(s)) \rangle_{L^2_x} \, ds,
\end{align}
where $\mathcal{R}$ denotes any rotation operator. (It is just $\mathcal{R}_{\alpha \beta}$ for interacting potential $V(x_{\alpha}-x_{\beta})$.) We want to control it by $\|V\|_{L_x^{\frac{q}{q-2}}} \|u\|_{V^p_{\Delta}}$.

Then, applying H\"older's inequality, the Strichartz estimate (Proposition \ref{mbStrichartz}) and the transference property of $V^p_{\Delta_x}$-space (Lemma \ref{trans}), we estimate it by
\begin{align}
& \sum_{k=1}^{K}\langle \phi_{k-1}, \tilde{w}(s_k)-\tilde{w}(s_{k-1})\rangle_{L^2_x} \\
&= \sum_{k=1}^{K}\langle  \phi_{k-1}, \tilde{w}(s_k)-\tilde{w}(s_{k-1})\rangle_{L^2_x} \\
&\lesssim \sum_{k=1}^{K} \| e^{it\Delta}\mathcal{R} \phi_{k-1}\|_{L^{q'}_tL^{q}_{x_{\alpha}}L^2_{\hat{x}_{\alpha}}}\|\textbf{1}_{[-s_k,-s_{k-1}]} \mathcal{R}(F(s))\|_{L^q_tL^{q'}_{x_{\alpha}}L^2_{\hat{x}_{\alpha}}} \textmd{ (using the H\"older) }  \\
&\lesssim \sum_{k=1}^{K} \| e^{it\Delta}\mathcal{R} \phi_{k-1}\|_{L^{q}_tL^{q}_{x_{\alpha}}L^2_{\hat{x}_{\alpha}}}\|\textbf{1}_{[-s_k,-s_{k-1}]} \mathcal{R}( F(s))\|_{L^q_tL^{q'}_{x_{\alpha}}L^2_{\hat{x}_{\alpha}}} \textmd{ (since it is a finine time interval) } \\
&\lesssim M^{\frac{d}{2}-\frac{d+2}{q}}\sum_{k=1}^{K} \|\phi_{k-1}\|_{L^2_x}  \|V\|_{L_x^{\frac{q}{q-2}}} \|\textbf{1}_{[-s_k,-s_{k-1}]} \mathcal{R}( u)\|_{L^q_tL^{q}_{x_{\alpha}}L^2_{\hat{x}_{\alpha}}} \textmd{ (using the Strichartz estimate. $M^{\textmd{power}}$ appears!) } \\
&\lesssim  \sum_{k=1}^{K} \|\phi_{k-1}\|_{L^2_x}  \|V\|_{L_x^{\frac{q}{q-2}}} \|\textbf{1}_{[-s_k,-s_{k-1}]} (u)\|_{V^p_{\Delta_x}} \textmd{ (using the transference principle. It fails!) } \\
&\lesssim  \|V\|_{L_x^{\frac{q}{q-2}}} \big\| \|\phi_{k-1}\|_{L^2_x} \big\|_{l^{p'}} \cdot \big\| \|\textbf{1}_{[-s_k,-s_{k-1}]} ( u)\|_{V^p_{\Delta_x}} \big\|_{l^{p}} \textmd{ (using the H\"older) } \\
&\lesssim  \|V\|_{L_x^{\frac{q}{q-2}}} \big\| \|\textbf{1}_{[-s_k,-s_{k-1}]} ( u)\|_{V^p_{\Delta_x}} \big\|_{l^{p}}.
\end{align}
The first equality (the first line equals the second line) follows from the addition of Fourier two supports, where $C$ is some positive constant. $q'$ is the dual of $q$ 
 in the sense of 
 \begin{equation}
  \frac{1}{q}+\frac{1}{q'}=1.   
 \end{equation}
 Moreover, $q$ and $\frac{q}{q-2}$ are the H\"older dual of $q'$ in the sense of
\begin{equation}
    \frac{1}{q}+\frac{1}{\frac{q}{q-2}}=\frac{1}{q'}.
\end{equation}
We note that in the last line we have used the inclusion property of discrete $L^p$ spaces (i.e. $l^{p}$-spaces). ($1<p<2$ implies $p'>2$.) 

To close the argument, now it remains to show that
\begin{equation}
\big\| \|\textbf{1}_{[-s_k,-s_{k-1}]} (u)\|_{V^p_{\Delta_x}} \big\|_{l^{p}}=\big\{ \sum_{k=1}^{K} \|\mathbf{1}_{[-s_k,-s_{k-1})}  u\|^{p}_{V^p_{\Delta_x}} \big\}^{\frac{1}{p}} \leq \| u\|_{V^p_{\Delta_x}}.    
\end{equation}
This estimate follows exactly as the Euclidean case (using the definition of $V^p_{\Delta_x}$). There is no difference in the periodic setting. \emph{The tentative proof ends here.}

 \textbf{Comments.} As we mentioned, the problem arises when we use the transference principle. This works surely for the Euclidean case. However, for the periodic case, it \textbf{fails} because of the essential derivative loss. (See Lemma \ref{trans}.) This argument fails also due to the $M$-power will not disappear such that the RHS can not be controlled, which leads the failure of the perturbative method.

 We also note that even for the one body case ($N=1$ in \eqref{maineq}), this method would still not work. The essential reason is still the derivative loss of Strichartz estimate.
\end{proof}

\section{The proof for Theorem 1.1}\label{4}

As discussed in the previous section, the perturbation method fails due to the derivative loss property for the periodic Strichartz estimate. Instead one needs to investigate more the property of the $H$-operator (treating the Laplacian and the interacting potentials as a whole). In this section, we present a new method rather than the perturbation method to obtain the Strichartz estimates (i.e. our main theorem, Theorem \ref{strimany}).

Before we state the proof for Theorem \ref{strimany}, we first mention three key elements as follows. As we will see very soon, they play crucial roles.

\emph{Key element 1:} We will use the equivalence of norms frequently: $\|(1+H)^{\epsilon} u\|_{L_x^2} \sim \|(1-\Delta)^{\epsilon} u\|_{L_x^2}$. ($H=-\Delta+V$).

\emph{Key element 2:} We will use the fact that we consider the Strichartz estimate on a \textbf{finite time interval}. It is important. For the Euclidean case, this scheme does not work since an infinite time interval is concerned for that case.

\emph{Key element 3:} We will use Strichartz estimate on tori (i.e. Lemma \ref{toriStrichartz}, see Bourgain-Demeter's seminal work \cite{BD} on Strichartz estimate via the decoupling method) as blackbox. This estimate is essential.

\begin{proof}[Proof of Theorem \ref{strimany}]
Recall that we have the assumption for the potential $V$, 
\begin{equation}\label{vcondition}
    |V(x)|+\left|\big((1-\Delta)^{\frac{\e}{2}} V\big)(x)\right|\le \frac{C}{N^2},
\end{equation} 
for some constant $C>0$.

If $V$ satisfies \eqref{vcondition}, it is then straightforward to check that 
$H_N=\Delta_x-V_N=\sum_{\alpha=1}^{N} \Delta_{x_{\alpha}}-\sum_{1\leq \alpha<\beta \leq N}V(x_{\alpha}-x_{\beta})$ is self adjoint and $C-H_N$ is a 
positive operator, and we also have  
\begin{equation}\label{equiv}
   C_0^{-1}\|(1-\Delta_{x}) u\|_{L^2(\mathbb{T}^{Nd})}\le \|(C-H_N) u\|_{L^2(\mathbb{T}^{Nd})}\le C_0 \|(1-\Delta_{x}) u\|_{L^2(\mathbb{T}^{Nd})},
\end{equation}
for some constant $C_0$.  By Stein's analytic interpolation theorem (see \cite{stein1956interpolation}),
this implies,
\begin{equation}\label{1}
    \|(C-H_N)^\e u\|_{L^2(\mathbb{T}^{Nd})}\approx  \|(1-\Delta_{x})^\e u\|_{L^2(\mathbb{T}^{Nd})},\,\,\,\forall\,\,0<\e\le 1.
\end{equation}

Recalling Proposition \ref{mbStrichartz}, we know  for $q=\frac{2(d+2)}{d}$, 
\begin{equation}\label{i.6}
\bigl\|e^{it\Delta_{x}}u_0\bigr\|_{L^q_{t,x_{\alpha}}(I\times \mathbb{T}^d)L^2_{\hat{x}_{\alpha}}} 
\lesssim \|u_0\|_{H^{\e}(\mathbb{T}^{Nd})},
\end{equation}
for arbitrarily small $\epsilon$-derivative loss. 

We note that our goal is to show for the same $q$, the following estimate can be recovered for $H_V$ operator.
\begin{equation}\label{i.61}
\bigl\|e^{itH_V}u_0\bigr\|_{L^q_{t,x_{\alpha}}(I\times \mathbb{T}^d)L^2_{\hat{x}_{\alpha}}} 
\lesssim \|u_0\|_{H^{\e}(\mathbb{T}^{Nd})}.
\end{equation}

By Duhamel's principle,
$$e^{itH_N}u_0=e^{it\Delta_{x}}u_0-i\sum_{1\leq \alpha<\beta \leq N}\int_0^t e^{i(t-s)\Delta_{x}}V(x_{\alpha}-x_{\beta})e^{isH_N}u_0ds.$$
The first term can be treated by using \eqref{i.6}.
For the second term, note that for each term in the summation,  by \eqref{i.6}, we have 
\begin{equation}\label{11}
    \begin{aligned}
     & \left\|   \int_0^t e^{i(t-s)\Delta_{x}}V(x_{\alpha}-x_{\beta})e^{isH_N}u_0ds\right\|_{L^q_{t,x_{\alpha}}(I\times \mathbb{T}^d)L^2_{\hat{x}_{\alpha}}} \\
     &\le \int_0^1 \left\|    e^{-is\Delta_{x}}V(x_{\alpha}-x_{\beta})e^{isH_N}u_0\right\|_{H^{\e}(\mathbb{T}^{Nd})} ds \\
     &=\int_0^1 \left\|V(x_{\alpha}-x_{\beta})e^{isH_N}u_0\right\|_{H^{\e}(\mathbb{T}^{Nd})} ds.
    \end{aligned}
\end{equation}

    By using the fractional Leibniz rule of   Kenig, Ponce, and Vega \cite{kenig1993well}, we have for each fixed $s$,
\begin{equation}\label{12}
    \begin{aligned}
    &\left\|V(x_{\alpha}-x_{\beta})e^{isH_N}u_0\right\|_{H^{\e}(\mathbb{T}^{Nd})}
\\
&\qquad\le \left\|V\right\|_{L^\infty(\mathbb{T}^{d})}  \left\|e^{isH_N}u_0\right\|_{H^\e(\mathbb{T}^{Nd})} \\
    & \qquad  \qquad+\left\|\big((1-\Delta)^{\frac{\e}{2}} V\big)\right\|_{L^\infty(\mathbb{T}^{d})}\left\|e^{isH_N}u_0\right\|_{L^2(\mathbb{T}^{Nd})}  +\left\|V\right\|_{L_{\e_1}^p(\mathbb{T}^{d})}  \left\|e^{isH_N}u_0\right\|_{L^q_{\e_2}(\mathbb{T}^{Nd})}.
    \end{aligned}
\end{equation}
Here $\frac1p+\frac1q=\frac12$, $1<p,q<\infty$, $\e_1+\e_2=\e$ with $0<\e_1,\e_2<1$, and for $1<p,q<\infty$, $\left\|f\right\|_{L^p_\e(\mathbb{T}^{d})}=\left\|(1-\Delta)^{\frac{\e}{2}} f\right\|_{L^p(\mathbb{T}^{d})}$   and $\left\|f\right\|_{L^q_\e(\mathbb{T}^{Nd})}=\left\|(1-\Delta_{x})^{\frac{\e}{2}} f\right\|_{L^q(\mathbb{T}^{Nd})}$ .

By Sobolev inequality, if we choose $q$ close enough to $2$, it is not hard to show 
\begin{equation}
     \left\|e^{isH_N}u_0\right\|_{L^q_{\e_2}(\mathbb{T}^{Nd})} \le   \left\|e^{isH_N}u_0\right\|_{H^{\e}(\mathbb{T}^{Nd})}.
\end{equation}

 By \eqref{1}, we also have 
\begin{equation}\label{eq}
\begin{aligned}
     \left\|e^{isH_N}u_0\right\|_{H^{\e}(\mathbb{T}^{Nd})}&\approx \left\|(C+H_N)^{\e} e^{isH_N}u_0\right\|_{L^2(\mathbb{T}^{Nd})}\\
     &=\left\|(C+H_N)^{\e} u_0\right\|_{L^2(\mathbb{T}^{Nd})}\approx \left\|(1-\Delta_{x})^{\e} u_0\right\|_{L^2(\mathbb{T}^{Nd})}.
\end{aligned}
\end{equation}

By \eqref{vcondition}, \eqref{11}, \eqref{12} and \eqref{eq}, we have 
\begin{equation}
    \left\|\sum_{1\leq \alpha<\beta \leq N}\int_0^t e^{i(t-s)\Delta_{x}}V(x_{\alpha}-x_{\beta})e^{isH_N}u_0ds\right\|_{L^q_{t,x_{\alpha}}(I\times \mathbb{T}^d)L^2_{\hat{x}_{\alpha}}}\le C \|u_0\|_{H^{\e}(\mathbb{T}^{Nd})},
\end{equation}
which completes the proof of Theorem \ref{strimany}.
\end{proof}
\section{Discussions for the nonlinear applications}\label{5}
In this section, we discuss some related nonlinear applications (well-posedness theory) based on the established Strichartz estimates. We will consider two cases: 4D cubic model and 3D quintic model since they are typical NLS models. 

\emph{On 4D cubic periodic NLS.}

First, we recall the 4D cubic NLS model\footnote{One may also consider 4D cubic NLS with a potential in the periodic setting, i.e. $(i\partial_t+\Delta_{x}+V(x))u=|u|^2u,\quad u(0,x)=u_0(x)\in H^1_{x}$ where $x\in \mathbb{T}^4$.} as follows,
\begin{equation}
(i\partial_t+\Delta_{x,y}+V(x-y))u=|u|^2u,\quad u(0,x,y)=u_0(x,y)\in H^1_{x,y},    
\end{equation}
where $x,y \in \mathbb{R}^2$.

We denote $H=\Delta_{x,y}+V(x-y)$.

After we have the Strichartz estimate for the operator $H$ as in Theorem \ref{strimany}, in general, there are several standard steps to establish the well-posedness theory: 1. Function spaces (Bourgain space $X^{s,b}$ and $V^p$-$U^p$ type spaces, for examples); 2. Bilinear estimate; 3. Nonlinear estimate; 4. The proof of the well-posedness theory.

We first present an easy way to show the well-posedness once we consider high regularity data\footnote{We note that we use the fact that the linear operator $e^{itH}$ is unitary.}. We consider the general case \eqref{maineq}. The calculations are as follows. 
For $s > \frac{dN}{2}$,
\begin{align*}
\norm{u}_{H^s} & \leq \norm{e^{it H} u_0}_{H^s}  + \norm{\int_0^t e^{i(t-s)H} \abs{u}^2 u (s) \, ds}_{H^s} \\
& \leq \norm{u_0}_{H^s} + \int_0^t \norm{ e^{i(t-s)H} \abs{u}^2 u (s) }_{H^s} \, ds\\
& \leq \norm{u_0}_{H^s} + \int_0^t \norm{ \abs{u}^2 u (s) }_{H^s} \, ds\\
& \leq \norm{u_0}_{H^s} + \int_0^t \norm{u}_{H^s}^3 \, ds \\
& \leq \norm{u_0}_{H^s} + T \norm{u}_{H^s}^3. 
\end{align*} 
By taking $T>0$ small enough, it yields
\begin{align*}
\norm{u}_{H^s} \leq 2 \norm{u_0}_{H^s}.
\end{align*}

Since the higher regularity assumption is too strong, that would be good if the well-posedness theory could be established in the energy space (i.e. $H^1$ space) since 4D cubic NLS and 3D quintic NLS are both energy critical. We consider the 4D cubic model. If one further assumes that the linear operator $e^{itH}$ commutes with the Littlewood-Pelay operator $P_N$\footnote{This assumption is a little too strong and not very natural.}, then the next lemma would follow from \cite{KV1}.
\begin{lemma}[Bilinear estimate]\label{bilinear}
Fix whole dimension $d \geq 3$ and $0<T<1$. Then for $1\leq N_2 \leq N_1$, we have
\begin{equation}\label{bilinear}
    \|u_{N_1} v_{N_2}\|_{L^2_{t,x}([0,T)\times  \mathbb{T}^{d} )}\lesssim N_2^{\frac{d-2}{2}}\|u_{N_1}\|_{Y^0([0,T))} \|v_{N_2}\|_{Y^0([0,T))}.
\end{equation}
\end{lemma}
We note that $Y^0$ space is based on the $H$-operator. The proof follows from Lemma 3.1 in \cite{KV1} once we have $e^{itH}$ commute with the Littlewood-Pelay operator. 

Then, the following nonlinear estimate holds in a standard way, which gives the well-posedness results for both of the two models: 4D cubic NLS and 3D quintic NLS. We refer to \cite{KV1} for more details.
\begin{lemma}\label{mainest}
Fix whole dimension $d=3,4$, for any $0<T<1$, we have

\begin{equation}\label{mainest1}
   \left\|\int_0^t e^{i(t-s)\Delta} F(u(s)) ds  \right\|_{X^1([0,T))} \lesssim \|u\|^{\frac{d+2}{d-2}}_{X^1([0,T))},
\end{equation}
and 
\begin{equation}\label{mainest2}
\aligned
    &\left\|\int_0^t e^{i(t-s)\Delta}  \left( F(u+w)(s)-F(u)(s) \right) ds  \right\|_{X^1([0,T))} \\
    &\lesssim \|w\|_{X^1([0,T))}\left(\|u\|_{X^1([0,T))}+\|w\|_{X^1([0,T))}\right)^{\frac{4}{d-2}}.
    \endaligned
\end{equation}
\end{lemma}

\emph{On 3D quintic periodic NLS.} The well-posedness theory would be similar to the 4D case with suitable modifications due to the quintic nonlinearity so we leave it for interested readers.

\section{Further remarks}\label{rem}
Finally, we would like to make a few additional remarks regarding the many-body Schr\"odinger model \eqref{maineq}.


1. An interesting direction is to consider many-body equations with nonlinearity $F(t,x_1,\ldots,x_N)$ and study their long-time behavior. There are few general theories and results regarding this topic, even for the Euclidean case, such as global well-posedness theory and long-time behavior. It may also be challenging to consider the general case, and the two-body case could still serve as a useful starting point. For instance, the $\Lambda$-equation in the Hartree-Fock-Bogoliubov equations is an example of the two-body case, although it is part of a coupled system, which makes it more complicated. (See \cite{chong2022global,chong2021global} for more details.) See Section \ref{5} for some discussions.


2. The results presented in this paper are focused solely on the estimates and PDE-level analysis of many-body Schr\"odinger equations. It would be interesting to study the many-body Schr\"odinger equations in the tori setting (or waveguide setting, as in \cite{zhao2023strichartz}) from the perspectives of mathematical physics. Examples of such perspectives include \cite{chong2019dynamical,grillakis2013pair,grillakis2017pair}. A natural question is to generalize the classical results on many body problems to the periodic case.


\bibliography{manybody}
\bibliographystyle{abbrv}

\end{document}